\documentclass[12pt]{amsart}

\usepackage{amsfonts}
\usepackage{amssymb}
\usepackage{amsthm}
\usepackage{amsmath}
\usepackage{MnSymbol}

\newtheorem{thm}{Theorem}[section]
\newtheorem{prop}[thm]{Proposition}
\newtheorem{cor}[thm]{Corollary}
\newtheorem{lem}[thm]{Lemma}

\newtheorem{defn}{Definition}

\newtheorem{conj}{Conjecture}
\newtheorem{ex}{Example}

\begin{document}

\title{Coding in graphs and linear orderings}

	\author{J.\ Knight, A.\ Soskova, and S.\ Vatev}

\address{Department of Mathematics\\
		University of Notre Dame\\
		USA}
	\email{knight.1@nd.edu}

	\address{Department of Mathematical Logic\\
		Sofia University\\
		Bulgaria}
	\email{asoskova@fmi.uni-sofia.bg}

\address{Department of Mathematical Logic\\
		Sofia University\\
		Bulgaria}
	\email{stefanv@fmi.uni-sofia.bg}
	\thanks{All 
three authors were partially supported by the NSF grant DMS-1600625.  The last two authors were partially supported by BNSF, MON, DN 02/16. }

\maketitle

\begin{abstract}

  There is a Turing computable embedding $\Phi$ of directed graphs $\mathcal{A}$ in undirected graphs (see \cite{M}).  Moreover, there is a fixed tuple of formulas that give a \emph{uniform} effective 
interpretation; i.e., for all directed graphs $\mathcal{A}$, these formulas interpret $\mathcal{A}$ in 
$\Phi(\mathcal{A})$. 
It follows that $\mathcal{A}$ is Medvedev reducible to $\Phi(\mathcal{A})$ uniformly; i.e., $\mathcal{A}\leq_s\Phi(\mathcal{A})$ with a fixed Turing operator that serves for all $\mathcal{A}$.  We observe that there is a graph $G$ that is not Medvedev reducible to any linear ordering.  Hence, $G$ is not effectively interpreted in any linear ordering.  Similarly, there is a graph that is not interpreted in any linear ordering using computable $\Sigma_2$ formulas.  Any graph can be interpreted in a linear ordering using computable $\Sigma_3$ formulas.  Friedman and Stanley \cite{FS} gave a Turing computable embedding $L$ of directed graphs in linear orderings.  We show that there is no fixed tuple of $L_{\omega_1\omega}$-formulas that, for all $G$, interpret the input graph $G$ in the output linear ordering $L(G)$.  Harrison-Trainor and Montalb\'{a}n \cite{HTM} have also shown this, by a quite different proof.                        

\end{abstract}

\section{Introduction} 

Friedman and Stanley \cite{FS} introduced Borel embeddings as a way of comparing classification problems for different classes of structures.  A Borel embedding of a class $\mathcal{K}$ in a class $\mathcal{K}'$ represents a uniform procedure for coding structures from $\mathcal{K}$ in structures from $\mathcal{K}'$.  Many Borel embeddings are actually Turing computable \cite{CCKM}.  A Turing computable embedding of a class $\mathcal{K}$ in a class $\mathcal{K}'$ represents an effective coding procedure.  

When $\mathcal{A}$ is coded in $\mathcal{B}$, effective decoding is represented by a Medvedev reduction of $\mathcal{A}$ to $\mathcal{B}$.  Harrison-Trainor, Melnikov, R. Miller, and Montalb\'{a}n \cite{HTM^3} defined a notion of effective interpretation of $\mathcal{A}$ in $\mathcal{B}$.  They also defined a notion of computable functor, where this is a pair of Turing operators, one taking copies of $\mathcal{B}$ to copies of $\mathcal{A}$, and the other taking isomorphisms between copies of $\mathcal{B}$ to isomorphisms between the corresponding copies of $\mathcal{A}$.  They showed that $\mathcal{A}$ is effectively interpreted in $
\mathcal{B}$ iff there is a computable functor from $\mathcal{B}$ to $\mathcal{A}$.  The first operator is a Medvedev reduction.  For some Turing computable embeddings $\Phi$, there are \emph{uniform} formulas that effectively interpret the input structure in the output structure, so we get a uniform Medvedev reduction.  This uniform Medvedev reduction represents uniform effective decoding.  Harrison-Trainor, R. Miller, and Montalb\'{a}n \cite{HTM^2} also considered interpretations by $L_{\omega_1\omega}$ formulas, guaranteeing Borel decoding.

We mention here a paper of Hirschfeldt, Khoussainov, Shore, and Slinko \cite{HKSS} that is related to \cite{HTM^3}.  The goal in \cite{HKSS} was to give conditions guaranteeing that computable dimension transfers from a structure $\mathcal{A}$ to a structure $\mathcal{B}$.  There are two main results.  In one, $\mathcal{A}$ is effectively defined in $\mathcal{B}$, while in the other, $\mathcal{A}$ is effectively interpreted in $\mathcal{B}$.  Both results involve conditions saying that $\mathcal{B}$ is essentially rigid over the defined (or interpreted) copy of $\mathcal{A}$.  

The class of undirected graphs and the class of linear orderings both lie on top under Turing computable embeddings.  The standard Turing computable embeddings of directed graphs (or structures for an arbitrary computable relational language) in undirected graphs come with uniform effective interpretations.  We give examples of graphs that are not Medvedev reducible to any linear ordering, or to the jump of any linear ordering.  By contrast, any single graph can be effectively interpreted in the second jump of some linear ordering.  

For the known Turing computable embedding of the class of graphs in the class of linear orderings, due to Friedman and Stanley, we show that there is no uniform interpretation defined by $L_{\omega_1\omega}$ formulas; that is, no fixed tuple of $L_{\omega_1\omega}$ formulas can interpret every graph in its Friedman-Stanley ordering.               

In the remainder of the introduction, we give some definitions and background.  We recall the Turing computable embedding of directed graphs (or structures for an arbitrary computable relational language) in undirected graphs.  For this embedding, we have a uniform effective interpretation.  In Section 2, we describe the graphs that are not Medvedev reducible to any linear ordering, or to the first jump of a linear ordering.  We explain why any graph is effectively interpreted in the second jump of some linear ordering---so we have a Medvedev reduction.  In Section 3, we recall the Turing computable embedding $\Phi$ of graphs in linear orderings due to Friedman and Stanley.  We show that there do not exist formulas of $L_{\omega_1\omega}$ that, for all $G$, interpret $G$ in~$\Phi(G)$.      

\subsection{Conventions}

We assume that the language of each structure is computable, where this means that the set of non-logical symbols is computable and we can effectively determine the type and arity of each symbol. We may assume that the languages are relational.  We restrict our attention to structures with universe equal to $\omega$.  Let $Mod(L)$ be the class of $L$-structures with this universe.  We identify a structure $\mathcal{A}$ with its atomic diagram $D(\mathcal{A})$.  We may identify this, via G\"{o}del numbering, with a set of natural numbers, or with an element of $2^\omega$.  Thus, we think of $Mod(L)$ as a subclass of $2^\omega$.  For a class of structures $\mathcal{K}\subseteq Mod(L)$, we suppose that $\mathcal{K}$ is axiomatized by an $L_{\omega_1\omega}$ sentence.  By a result of L\'{o}pez-Escobar \cite{LE}, this is the same as assuming that $\mathcal{K}$ is a Borel subclass of $Mod(L)$ closed under isomorphism.  

\subsection{Borel embeddings}

The following definition is from \cite{FS}.  

\begin{defn}  

We say that a class $\mathcal{K}$ is \emph{Borel embeddable} in a class $\mathcal{K}'$, and we write $\mathcal{K}\leq_B \mathcal{K}'$, if there is a Borel function $\Phi:\mathcal{K}\rightarrow \mathcal{K}'$ such that for $\mathcal{A},\mathcal{B}\in \mathcal{K}$, $\mathcal{A}\cong\mathcal{B}$ iff $\Phi(\mathcal{A})\cong\Phi(\mathcal{B})$.

\end{defn}

A Borel embedding of $\mathcal{K}$ into $\mathcal{K}'$ represents a uniform procedure for coding structures from $\mathcal{K}$ in structures from $\mathcal{K}'$.  Friedman and Stanley \cite{FS} gave the following result.            

\begin{thm}

The following classes lie on top under $\leq_B$.

\begin{enumerate}

\item  undirected graphs

\item  fields of any fixed characteristic

\item  $2$-step nilpotent groups

\item  linear orderings

\end{enumerate}

\end{thm}

Friedman and Stanley defined an embedding of graphs in fields of any fixed characteristic.  They also defined an embedding of graphs in linear orderings.  For the other classes listed above, Friedman and Stanley credit earlier sources.  Lavrov \cite{L} defined an embedding of $Mod(L)$ in undirected graphs, for any $L$.  There are similar constructions due to Nies \cite{N} and Marker \cite{M}.  Mekler \cite{Me} defined an embedding of graphs in $2$-step nilpotent groups.  Alternatively, we get an embedding of graphs in $2$-step nilpotent groups by composing the embedding of graphs in fields with an earlier embedding by Mal'tsev \cite{Ma} of fields in $2$-step nilpotent groups.  

\subsection{Turing computable embeddings}

Kechris suggested to the first author that she and her students should consider effective embeddings.  This is done in \cite{CCKM}, \cite{KMV}.  

\begin{defn} 

We say that a class $\mathcal{K}$ is \emph{Turing computably embedded} in a class $\mathcal{K}'$, and we write $\mathcal{K}\leq_{tc} \mathcal{K}'$, if there is a Turing operator $\Phi:\mathcal{K}\rightarrow \mathcal{K}'$ such that for all $\mathcal{A},\mathcal{B}\in \mathcal{K}$, $\mathcal{A}\cong\mathcal{B}$ iff $\Phi(\mathcal{A})\cong\Phi(\mathcal{B})$.

\end{defn}

A Turing computable embedding represents an effective coding procedure.  The next result is in \cite{CCKM}.  

\begin{thm}

The following classes lie on top under $\leq_{tc}$.

\begin{enumerate}

\item  undirected graphs

\item  fields of any fixed characteristic

\item  $2$-step nilpotent groups

\item  linear orderings

\end{enumerate} 

\end{thm}

The involves simply noting that the Borel embeddings of Friedman-Stanley, Lavrov, Nies, Marker, Mekler, and Mal'tsev are all, in fact, Turing computable. 

\subsection{Medvedev reductions}

A \emph{problem} is a subset of $2^\omega$ or $\omega^\omega$.  Problem $P$ is Medvedev reducible to problem $Q$ if there is a Turing operator $\Phi$ that takes elements of $Q$ to elements of $P$.  The problems that interest us ask for copies of particular structures, where each copy is identified with an element of $2^\omega$.  

\begin{defn}

We say that $\mathcal{A}$ is \emph{Medvedev reducible} to $\mathcal{B}$, and we write $\mathcal{A}\leq_s\mathcal{B}$ if there is a Turing operator that takes copies of $\mathcal{B}$ to copies of $\mathcal{A}$.  

\end{defn}

Supposing that $\mathcal{A}$ is coded in $\mathcal{B}$, a Medvedev reduction of $\mathcal{A}$ to 
$\mathcal{B}$ represents an effective decoding procedure.    

\subsection{Sample embedding}

Below, we describe Marker's Turing computable embedding of directed graphs in undirected graphs. 

\begin{enumerate}

\item  For each point $a$ in the directed graph $\mathcal{A}$, the undirected graph $\mathcal{B}$ has a point $b_a$ connected to a triangle. 

\item  For each ordered pair of points $(a,a')$ from $\mathcal{A}$, $\mathcal{B}$ has a point $p_{(a,a')}$ that is connected directly to $b_a$ and with one intermediate point to $b_{a'}$ ($p_{(a,a')}$ and $b_{a'}$ are each connected directly to the intermediate point $c$).  The point $p_{(a,a')}$ is connected to a square if there is an arrow from $a$ to $a'$, and to a pentagon otherwise.

\end{enumerate} 

For structures $\mathcal{A}$ with more relations, the same idea works---we use more special points and more $n$-gons. 

\bigskip
\noindent
\textbf{Fact}:  For Marker's embedding $\Phi$ of directed graphs in undirected graphs, there are finitary existential formulas that, for all inputs $\mathcal{A}$, define in $\Phi(\mathcal{A})$ the following:  

\begin{enumerate}

\item  the set $D$ consisting of the points $b_a$ connected to a triangle, 

\item  the set of ordered pairs $(b_a,b_{a'})$ such that the special point $p_{(a,a')}$ is connected to a square,

\item  the set of ordered pairs $(b_a,b_{a'})$ such that the special point $p_{(a,a')}$ is connected to a pentagon.

\end{enumerate}

This guarantees that any copy of $\Phi(\mathcal{A})$ computes a copy of $\mathcal{A}$.  

\subsection{Effective interpretations and computable functors}

In a number of familiar examples where $\mathcal{A}\leq_s\mathcal{B}$, the structure $\mathcal{A}$ is defined or interpreted in $\mathcal{B}$ using formulas of special kinds.    

\begin{ex}

The usual definition of the ring of integers $\mathbb{Z}$ involves an interpretation in the semi-ring of natural numbers $\mathbb{N}$.  Let $D$ be the set of ordered pairs $(m,n)$ of natural numbers.  We think of the pair $(m,n)$ as representing the integer $m - n$.  With this in mind, we can easily give finitary existential formulas that define ternary relations of addition and multiplication on $D$, and the complements of these relations, and a congruence relation $\sim$ on $D$, and the complement of this relation, such that $(D,+,\cdot)/_{\sim}\cong\mathbb{Z}$.      

\end{ex}

Harrison-Trainor, Melnikov, R.\ Miller, and Montalb\'{a}n \cite{HTM^3} defined a very general kind of interpretation of $\mathcal{A}$ in $\mathcal{B}$ guaranteeing that $\mathcal{A}\leq_s\mathcal{B}$.  The tuples in $\mathcal{B}$ that represent elements of $\mathcal{A}$ have no fixed arity.  Recall that a \emph{computable $\Sigma_1$ formula} is a c.e.\ disjunction of finitary existential formulas.  We will use $\Sigma^c_\alpha$ to denote the computable infinitary $\Sigma_\alpha$ formulas, and the same for $\Delta^c_\alpha$ and $\Pi^c_\alpha$.  Normally, we consider formulas with a fixed tuple of variables.  However, following \cite{HTM^3}, we will consider relations $R\subseteq\mathcal{B}^{<\omega}$ in our interpretations, and we will say that such a relation $R$ is defined in $\mathcal{B}$ by a $\Sigma_1^c$ formula that when there is a computable sequence of $\Sigma_1^c$ formulas $\varphi_n(\bar{x}_n)$ defining $R\cap\mathcal{B}^n$.  Our $\Sigma_1^c$ definition of $R$ is $\bigvee_n \varphi_n(\bar{x}_n)$.  A relation $R$ defined in this way is c.e.\ relative to $\mathcal{B}$.      

\begin{ex}

The dependence relation on tuples in a $\mathbb{Q}$-vector space is a familiar relation with no fixed arity.  It is defined by a $\Sigma_1^c$ formula $\bigvee_n\varphi_n(\bar{x}_n)$ of the kind that we use for effective interpretations.  We let $\varphi_n(\bar{x}_n) = \bigvee_\lambda \lambda(\bar{x}_n) = 0$, where $\lambda$ ranges over the non-trivial rational linear combinations of $\bar{x}_n = (x_1,\ldots,x_n)$.     

\end{ex}        

In a given structure $\mathcal{B}$, we say that a relation $R$ is \emph{$\Delta_1^c$-definable over $\bar{c}$} if $R$ and the complementary relation $\neg{R}$ are each defined by $\Sigma^c_1$ formulas, with parameters $\bar{c}$.                 

\begin{defn}

A structure $\mathcal{A} = (A,R_i)$ is \emph{effectively interpreted} in a structure $\mathcal{B}$ if there is a set $D\subseteq\mathcal{B}^{<\omega}$, $\Sigma^c_1$-definable over $\emptyset$, and there are relations $\sim$ and $R_i^*$ on $D$, $\Delta_1^c$-definable over $\emptyset$, such that $(D,R_i^*)/_\sim\cong\mathcal{A}$.   

\end{defn} 

Above, we described Marker's Turing computable embedding of directed graphs in undirected graphs, and we saw there are uniform finitary existential formulas that in the output directed graph a set $D$ and relations $\pm R^*$ such that $(D,R^*)$ is isomorphic to the input undirected graph.  Friedman and Stanley's original embedding of graphs in fields involved a uniform interpretation by means of $\Sigma^c_3$ formulas.  A more recent embedding of graphs in fields, due to R. Miller, Poonen, Schoutens, and Shlapentokh \cite{MPSS}, gives a uniform effective interpretation.

\bigskip

Harrison-Trainor, Melnikov, R. Miller, and Montalb\'{a}n \cite{HTM^3} defined a second notion.         

\begin{defn}

A \emph{computable functor} from $\mathcal{B}$ to $\mathcal{A}$ is a pair of Turing operators $(\Phi,\Psi)$ such that  $\Phi$ takes copies of $\mathcal{B}$ to copies of $\mathcal{A}$ and $\Psi$ takes isomorphisms between copies of $\mathcal{B}$ to isomorphisms between the corresponding copies of $\mathcal{A}$, so as to preserve identity and composition. 

\end{defn} 

The main result from \cite{HTM^3} gives the equivalence of the two notions.

\begin{thm}
\label{equiv}

For structures $\mathcal{A}$ and $\mathcal{B}$, $\mathcal{A}$ is effectively interpreted in $\mathcal{B}$ iff there is a computable functor $\Phi,\Psi$ from $\mathcal{B}$ to $\mathcal{A}$.

\end{thm} 

\begin{cor}

If $\mathcal{A}$ is effectively interpreted in $\mathcal{B}$, then $\mathcal{A}\leq_s\mathcal{B}$.

\end{cor}

\begin{proof}

We get a Medvedev reduction by taking the first half $\Phi$ of the computable functor $\Phi,\Psi$. 
\end{proof}

Kalimullin \cite{K} showed that the converse of the corollary fails.  We may have a Turing operator $\Phi$ taking copies of $\mathcal{B}$ to copies of $\mathcal{A}$ without having a Turing operator $\Psi$ taking triples $(\mathcal{B}_1,\mathcal{B}_2,f)$ to $g$, where $\mathcal{B}_1,\mathcal{B}_2$ are copies of $\mathcal{B}$ and $\mathcal{B}_1\cong_f\mathcal{B}_2$ and $\Phi(\mathcal{B}_1)\cong_g\Phi(\mathcal{B}_2)$.  

\bigskip

In the proof of Theorem \ref{equiv}, it is important that the set $D$ in the interpretation consist of tuples from $\mathcal{B}$ of arbitrary arity.  The same is true in the proof of the following.   

\begin{prop}\label{pr:comp-effective-int}

If $\mathcal{A}$ is computable, then $\mathcal{A}$ is effectively interpreted in all structures $\mathcal{B}$.

\end{prop}

\begin{proof}

Let $D = \mathcal{B}^{<\omega}$.  We let $\bar{b}\sim\bar{c}$ if $\bar{b},\bar{c}$ are tuples of the same length.  For simplicity, suppose $\mathcal{A} = (\omega,R)$, where $R$ is binary.  We define $R^*$ such that $R^*(\bar{b},\bar{c})$ holds iff $\mathcal{A}\models R(m,n)$, where $m$ is the length of $\bar{b}$ and $n$ is the length of $\bar{c}$.  
\end{proof}

\subsection{Interpretations by more general formulas}

We may consider interpretations of $\mathcal{A}$ in $\mathcal{B}$, where $D$, $\pm\sim$, and 
$\pm R_i^*$ are defined in $\mathcal{B}$ by $\Sigma^c_2$ formulas, and we have 
$(D,(R_i^*)_{i\in\omega})/_{\sim}\cong\mathcal{A}$.  There is a notion of \emph{jump} of a structure, which was independently rediscovered by a number of researchers, see, e.g. \cite{Baleva06}, \cite{SS}, \cite{Mon}, and \cite{Stukachev}.
Here we follow Montalb\'an's approach \cite{Mon, Mo}, which is the most suitable for our purposes.  The \emph{jump} of $\mathcal{A}$ is a structure $\mathcal{A}' = (\mathcal{A},(R_i)_{i\in\omega})$, where $R_i$ is the relation defined in $\mathcal{A}$ by the $i^{th}$ $\Sigma^c_1$ formula.  We can iterate the jump, forming $\mathcal{A}'' = (\mathcal{A}')'$, etc.  For our purposes, the following facts about jumps suffice.  

\begin{enumerate}

\item  For a structure $\mathcal{A}$, the \emph{jump} is a structure $\mathcal{A}'$ such that the relations defined in $\mathcal{A}'$ by $\Sigma^c_1$ formulas are just those defined in $\mathcal{A}$ by $\Sigma^c_2$ formulas.  

\item  For a structure $\mathcal{A}$, the jump structure $\mathcal{A}'$ is computed by $D(\mathcal{A})'$.  

\item  The relations defined in $\mathcal{A}''$ by $\Sigma^c_1$ formulas are just those defined in $\mathcal{A}$ by $\Sigma^c_3$ formulas.

\end{enumerate}              

Harrison-Trainor, R. Miller, and Montalb\'{a}n \cite{HTM^2} proved the analogue of the result 
from \cite{HTM^3} in which the interpretations are defined by formulas of $L_{\omega_1\omega}$, and the functors are Borel.  Again for an interpretation of $\mathcal{A}$ in $\mathcal{B}$, the set of tuples in $\mathcal{B}$ that represent elements of $\mathcal{B}$ may have arbitrary arity.  If  
$R\subseteq\mathcal{B}^{<\omega}$, and we have a countable sequence of $L_{\omega_1\omega}$-formulas $\varphi_n(\bar{x}_n)$ defining $R\cap\mathcal{B}^n$, then we refer to $\bigvee_n\varphi_n(\bar{x}_n)$ as an $L_{\omega_1\omega}$ definition of $R$.       

\begin{thm}

A structure $\mathcal{A}$ is interpreted in $\mathcal{B}$ using $L_{\omega_1\omega}$-formulas iff there is a Borel functor $(\Phi,\Psi)$ from $\mathcal{B}$ to $\mathcal{A}$. 

\end{thm} 

\section{Interpreting graphs in linear orderings}

As we have seen, any structure can be effectively interpreted in a graph.  Linear orderings do not have so much interpreting power.  To show this, we use the following result of Linda Jean Richter \cite{R}.

\begin{prop} [Richter]
\label{Richter}

For a linear ordering $L$, the only sets computable in all copies of $L$ are the computable sets.

\end{prop} 

\begin{prop}  
\label{graph.not.interpreted} 

There is a graph $G$ such that for all linear orderings $L$, $G\not\leq_s L$.

\end{prop}

\begin{proof}

Let $S$ be a non-computable set.  Let $G$ be a graph such that every copy computes $S$.  We may take $G$ to be a ``daisy'' graph, consisting of a center node with a ``petal'' of length $2n+3$ if $n\in S$ and $2n+4$ if $n\notin S$.  Now, apply Proposition \ref{Richter}.   
\end{proof} 

The following result, from \cite{K1}, is a lifting of Proposition \ref{Richter}.                     

\begin{prop}
\label{Jump}

For a linear ordering $L$, the only sets computable in all copies of $L'$ (or in the jumps of all copies of $L$) are the $\Delta^0_2$ sets.

\end{prop}

This yields a lifting of Proposition \ref{graph.not.interpreted}.    

\begin{prop}  

There is a graph $G$ such that for all linear orderings $L$, $G\not\leq_s L'$.  

\end{prop}  

\begin{proof}

Let $S$ be a non-$\Delta^0_2$ set.  Let $G$ be a graph such that every copy computes $S$.  Then apply Proposition \ref{Jump}.    
\end{proof}

The pattern above does not continue.  The following is well-known (see Theorem 9.12 \cite{AK}).

\begin{prop}
\label{two.jumps}

For any set $S$, there is a linear ordering $L$ such that for all copies of $L$, the second jump computes $S$.  

\end{prop}

\begin{proof} [Proof sketch]

For a set $A$, the ordering $\sigma(A\cup\{\omega\})$ (the ``shuffle sum'' of orderings of type $n$ for $n\in A$ and of type $\omega$) consists of densely many copies of each of these orderings.  The degrees of copies of $\sigma(A\cup\{\omega\})$ are the degrees of sets $X$ such that $A$ is c.e.\ relative to $X''$.  Let $A = S\oplus S^c$, where $S^c$ is the complement of $S$.  Consider the linear ordering $L = \sigma(A\cup\{\omega\})$. Then we have a pair of finitary $\Sigma_3$ formulas saying that $n \in S$ iff $L$ has a maximal discrete set of size $2n$ and $n \not \in S$ iff $L$ has a maximal discrete set of size $2n+1$.  It follows that any copy of $L''$ uniformly computes the set $S$.
\end{proof}

Using Proposition \ref{two.jumps}, we get the following.    

\begin{prop}

For any graph $G$, there is a linear ordering $L$ such that $G\leq_s L''$, 

\end{prop}

\begin{proof}

Let $S$ be the diagram of a specific copy of $G$ and let $L$ be as in Proposition \ref{two.jumps}.  Then $G\leq_s L''$.     
\end{proof}

\section{Turing computable embedding of graphs in linear orderings}

The class of linear orderings, like the class of graphs, lies on top under Turing computable embeddings.  We describe the Turing computable embedding $L$, given in \cite{FS}, of directed graphs in linear orderings.  

\bigskip
\noindent
\textbf{Friedman-Stanley embedding}.  First, let $(A_n)_{n\in\omega}$ be an effective partition of $\mathbb{Q}$ into disjoint dense sets.  Let $(t_n)_{1\leq n<\omega}$ be a list of the atomic types in the language of directed graphs.  We let $t_1$ be the type of $\emptyset$, we put the types for single elements next, then the types for distinct pairs, then the types for distinct triples, etc.  For a graph $G$, the ordering $L(G)$ is a sub-ordering of $\mathbb{Q}^{<\omega}$, with the lexicographic ordering.  The elements of $L(G)$ are the finite sequences 
\[\sigma = r_0q_1r_1\ldots r_{n-1}q_nr_nk\in\mathbb{Q}^{<\omega}\] 
satisfying the following conditions:  

\begin{enumerate}

\item  for each $i < n$, $r_i\in A_0$, and $r_n\in A_1$, 

\item  there is a special tuple in $G$ associated with $\sigma$, of length $n$, satisfying the atomic type $t_m$, 

\item  if $n = 0$, then the special tuple is $\emptyset$, while if $n\geq 1$, then the special tuple has form $a_1,\ldots,a_n$, where for all $i$ with $1\leq i\leq n$, $q_i\in A_{a_i}$---we can read off the special tuple from the terms in $\sigma$,  

\item  $k$ is a natural number less than $m$.

\end{enumerate}

In talks, the first author has claimed, without any proof, that this embedding does not represent an interpretation.  Our goal in the rest of the paper is to prove the following theorem.    

\begin{thm} [Main Theorem]
\label{main}

There do not exist $L_{\omega_1\omega}$-formulas that, for all graphs $G$, interpret $G$ in $L(G)$.

\end{thm}

We begin with some definitions and simple lemmas about $L(G)$.  

\begin{defn}

Let $b = r_0q_1r_1\ldots r_{n-1}q_nr_nk\in L(G)$.  We say that $b$ \emph{mentions} $\bar{a}$ if $\bar{a}$ is the special tuple in $G$ of length $n$, such that for $1\leq i\leq n$, $q_i\in A_{a_i}$.   

\end{defn}

\begin{lem}

Suppose $b\in L(G)$ mentions $\bar{a}$.  Then $b$ lies in a maximal discrete interval of some finite size $m\geq 1$.  The number $m$ tells us the atomic type of $\bar{a}$; in particular, it tells us the length of $\bar{a}$.  

\end{lem}

\begin{proof}

It is clear from the definition of $L(G)$ that if $b$ mentions $\bar{a}$, where $\bar{a}$ satisfies the atomic type $t_m$ on our list, then $b$ lies in a maximal discrete set of size $m$.  Knowing just that $b$ lies in a maximal discrete set of size $m$, we know the atomic type, and this tells us the length of~$\bar{a}$.   
\end{proof}

The structure of the linear ordering $L(G)$ does not directly tell us the lengths of the elements $b$ (as elements of $\mathbb{Q}^{<\omega}$).  However, if $b$ mentions $\bar{a}$ of length $n$, then $b$ has length $2n+2$.  

\begin{lem}

If $b\in L(G)$ has length $2n+2$, then there is an infinite interval around $b$ that consists entirely of elements of length at least $2n+2$.  

\end{lem}

\begin{proof}    

Suppose that $b = r_0q_1r_1\ldots r_{n-1}q_n r_n k$.  The elements $d$ that extend the initial segment
$r_0q_1r_1\ldots r_{n-1}q_n$, of length $2n$, are closer to $b$ than those that differ on one of the first $2n$ terms.  These $d$ all have length at least $2n+2$, and they form the interval we want.
\end{proof}

\begin{lem}

Let $b,b'\in L(G)$, where $b < b'$, and let $d$ be an element of $[b,b']$ of minimum length.  If $d$ mentions $\bar{c}$, then all elements of $[b,b']$ mention extensions of $\bar{c}$.  

\end{lem}     

\begin{proof}

Say that $d$ has length $2k+2$.  Then $b$ and $b'$ are both in an interval around $d$ consisting of elements of length at least $2k+2$.  Let $\sigma$ be the initial segment of $d$ of length $2k$.  Then all elements of $[b,b']$ must extend $\sigma$.  Thus, all of these mention extensions of $\bar{c}$. 
\end{proof}

Let $\bar{b}$ be a tuple in $L(G)$.  For each $b_i$ in $\bar{b}$, let $\bar{a}_i$ be the tuple in $G$ mentioned by $b_i$.  The formulas true of $\bar{b}$ in $L(G)$ are determined by the formulas true in $G$ of the various $\bar{a}_i$, together with the ``shape'' of~$\bar{b}$.     

\begin{defn}

For a tuple $\bar{b} = (b_1,\ldots,b_n)$ in $L(G)$, the shape encodes the following information:

\begin{enumerate}

\item  the order type of $\bar{b}$---for simplicity, we suppose that \\
$b_1 < b_2 < \ldots < b_n$,

\item  the size of each interval $(b_i,b_{i+1})$---we note that the interval is infinite unless $b_i,b_{i+1}$ belong to the same finite discrete set in $L(G)$, which means that they agree on all but the last term,

\item the location of each $b_i$ in the finite discrete interval to which it belongs,

\item  the length of each $b_i$,

\item  for $i < n$, the number $k_i$ such that $2k_i+2$ is the length of a shortest element $d$ in the interval $[b_i,b_{i+1}]$---$d$ mentions a tuple $\bar{c}$ of length $k_i$, and all elements of $[b_i,b_{i+1}]$ mention tuples that extend~$\bar{c}$.   

\end{enumerate}          

\end{defn}

\begin{prop}

For each $n$-tuple $\bar{b}$, there exist $\Pi^c_4$, and $\Sigma^c_4$, formulas in the language of linear orderings, that, for all $G$, say in $L(G)$ that the $n$-tuple $\bar{x}$ has the same shape as $\bar{b}$.
\end{prop}

\begin{proof}

We note the following:   
  
\begin{enumerate}

\item  For any finite $n$, we have a finitary $\Sigma_2$ formula saying of an interval that it has at least $n$ elements and it does not have at least $n+1$ elements.  Thus, there are finitary $\Sigma_2$ and $\Pi_2$ formulas  saying that an interval $(b_i,b_{i+1})$ has size $n$.   

\item  We have a finitary $\Sigma_3$ formula saying that $b_i$ sits in a specific position in a maximal discrete set of size $n$.  

\item  Assuming that our list of the atomic types $(t_n)_{1\leq n < \omega}$ is as described above, we have finitary $\Sigma_3$ formulas saying that $b_i$ has length $2n+2$---we take a finite disjunction of formulas saying that $b_i$ lies in a maximal discrete interval of size $r$, where $t_r$ is the atomic type of a tuple of length $n$. 

\item  For each $k$, we have a finitary $\Pi_3$ formula saying that all $z$ in $[b_i,b_{i+1}]$ have length at least $2k+2$.

\end{enumerate}

Taking an appropriate finite conjunction of the formulas described above, we obtain a $\Sigma^c_4$ definition of the set of tuples of a specific shape, and also a $\Pi^c_4$ definition.  
\end{proof}

\noindent
\textbf{Remarks on elements of length $2$}:  Suppose $d$ has length $2$.  Then $\emptyset$ is the tuple mentioned by $d$ and the atomic type of $\emptyset$ is $t_1$, so $d$ has the form $r_00$, where $r_0\in A_1$.  Note that $d$ is the only element of $L(G)$ that starts with $r_0$.  If $b < d < b'$, then $b$ has first term $r$ and $b'$ has first term $r'$, where $r < r_0 < r'$.  Since all $A_i$ are dense in $\mathbb{Q}$, essentially everything happens in the intervals $(b,d)$ and $(d,b')$.

\begin{lem}
\label{lem10}

Suppose $c < c^* < c'$ in $L(G)$, where $c^*$ has length $2$.     

\begin{enumerate}

\item  For any $\bar{e}$ in $(c,\infty)$, there is an automorphism of $(c,\infty)$ taking $\bar{e}$ to some $\bar{e}'$ in the interval $(c,c^*)$.  

\item  For any $\bar{e}$ in $(-\infty,c')$, there is an automorphism of $(-\infty,c')$ taking $\bar{e}$ to some $\bar{e}'$ in the interval $(c^*,c')$. 

\end{enumerate}  

\end{lem}

\begin{proof}

We prove (1).  Note that $c^*$ has form $r0$, where $r\in A_1$.  The first term of $c$ is some $q < r$.  Let $c^{**}$ be an element of length $2$ to the right of all elements of $\bar{e}$, with first term $p$.  There is a permutation of $\mathbb{Q}$, say $f$, such that 
\begin{enumerate}

\item [(a)]  $f$ preserves the ordering and membership in the $A_i$'s (i.e., $f$ is an automorphism of the structure 
$(\mathbb{Q},<,(A_i)_{i\in\omega})$, 

\item [(b)]  $f(q) = q$ and $f(p) = r$.  

\end{enumerate}
We define an automorphism $g$ of $(c,\infty)$, taking each element $x\sigma$ to $f(x)\sigma$---we are changing just the first term.  The fact that $f$ preserves the ordering and membership in $A_i$'s is needed to be sure that $g$ has domain and range $(c,\infty)$.         
\end{proof}

If $a < b$ in the ordering $L(G)$, we may say that $a$ lies \emph{to the left of} $b$, or that $b$ lies \emph{to the right of} $a$.

\begin{lem}
\label{lem11}

Let $\bar{b}$ be a finite tuple in $L(G)$, and let $c$ be an element of $L(G)$.   

\begin{enumerate}

\item  There is an automorphism of $L(G)$ taking $\bar{b}$ to a tuple $\bar{b}'$, with $<$-first element $b'$, such that $c < b'$ and the interval $(c,b')$ contains some element of length $2$.   

\item  There is also an automorphism taking $\bar{b}$ to a tuple $\bar{b}''$, with $<$-greatest element $b''$, such that $b'' < c$ and the interval $(b'',c)$ contains some element of length $2$.    

\end{enumerate}

\end{lem}

\begin{proof}   

We give the proof for (1).  Suppose that $c$ (as a finite sequence) begins with $r$.  Let $b$ be the $<$-first element of $\bar{b}$, and suppose that $b$ (as a finite sequence) begins with $p$.  Let $f$ be a permutation of $\mathbb{Q}$ that preserves the ordering and membership in the $A_i$'s, and such that $f(p) > r$.  We have an automorphism $g$ of $L(G)$ such that 
\[g(x\sigma) = f(x)\sigma.\]  
By our choice of $f$, it follows that $g$ has domain and range all of $L(G)$.  To see that there is an element of length $2$ between $c$ and the $<$-first element of $g(\bar{b})$, we note that there is an element of $A_1$ between $r$ and $f(p)$.            
\end{proof}

\subsection{The relations $\sim^\gamma$}

Below, we recall a family of equivalence relations, defined for pairs of tuples, from the same structure, or from two different structures.    
  
\begin{defn}

Let $\mathcal{A}$ and $\mathcal{B}$ be structures for a fixed finite relational language.  Let $\bar{a}$ and 
$\bar{b}$ be tuples of the same length, where $\bar{a}$ in $\mathcal{A}$ and $\bar{b}$ is in $\mathcal{B}$.  

\begin{enumerate}

\item  $(\mathcal{A},\bar{a})\sim^0(\mathcal{B},\bar{b})$ if the tuples $\bar{a}$ and $\bar{b}$ satisfy the same atomic formulas in their respective structures.

\item  For $\gamma > 0$, $(\mathcal{A},\bar{a})\sim^\gamma(\mathcal{B},\bar{b})$ if for all $\beta < \gamma$, 

\begin{enumerate}

\item  for all $\bar{c}\in \mathcal{A}$, there exists $\bar{d}\in \mathcal{B}$ such that \\
$(\mathcal{A},\bar{a},\bar{c})\sim^\beta(\mathcal{B},\bar{b},\bar{d})$,

\item  for all $\bar{d}\in \mathcal{B}$, there exists $\bar{c}\in \mathcal{A}$ such that \\
$(\mathcal{A},\bar{a},\bar{c})\sim^\beta(\mathcal{B},\bar{b},\bar{d})$.

\end{enumerate}

\end{enumerate}

\end{defn}

\noindent
\textbf{Note}:  We write $\mathcal{A}\sim^\gamma\mathcal{B}$ to indicate that 
$(\mathcal{A},\emptyset)\sim^\gamma(\mathcal{B},\emptyset)$.

\begin{lem}
\label{lem:equivalence:characterization}
  
Let $\mathcal{A}$ be a computable structure for a finite relational language.  For any 
$\gamma < \omega^{CK}_1$ and for any tuple $\bar{a}$ in $\mathcal{A}$, we can effectively find a $\Pi^c_{2\gamma}$-formula $\varphi^\gamma_{\bar{a}}(\bar{x})$ such that 
$\mathcal{A} \models \varphi^\gamma_{\bar{a}}(\bar{b})$ iff $\bar{a} \sim^\gamma \bar{b}$.

\end{lem}

\begin{proof}

We proceed by induction on $\gamma$.  Let $\gamma = 0$. Then
\[\varphi^0_{\bar{a}}(\bar{x}) = \bigwedge_{\phi(\bar{x}) \in B} \phi(\bar{x}),\]
where $B$ is the set of atomic formulas and negations of atomic formulas true of $\bar{a}$ in $\mathcal{A}$.  This formula is finitary quantifier-free.
Suppose $\gamma > 0$, where we have the formulas $\varphi^\beta_{\bar{a}}$ for all $\beta < \gamma$ and all $\bar{a}$.  Then
\[\varphi^\gamma_{\bar{a}}(\bar{x}) = \bigwedge_{\beta < \gamma} 
[\bigwedge_{\bar{c}}(\exists \bar{y})\varphi^\beta_{\bar{a},\bar{c}}(\bar{x},\bar{y})\ \&\ \bigwedge_{\bar{y}} 
(\forall \bar{y})\bigvee_{\bar{c}} \varphi^\beta_{\bar{a},\bar{c}}(\bar{x},\bar{y})]\]
This formula is $\Pi^c_{2\gamma}$, as required.  
\end{proof}

\begin{lem}
\label{lem:equivalence:language}

Let $L$ be a fixed finite relational language.  For any computable ordinal $\gamma$,
and any tuples of variables $\bar{x}$, $\bar{y}$, of the same length, we can effectively find a computable $\Pi_{2\gamma}$-formula $\varphi^\gamma(\bar{x},\bar{y})$ such that for any $L$-structure $\mathcal{A}$, and any tuples $\bar{a}$ and $\bar{b}$ from $\mathcal{A}$, $\mathcal{A}\models \varphi^\gamma(\bar{a},\bar{b})$ iff $(\mathcal{A},\bar{a}) \sim^\gamma (\mathcal{A},\bar{b})$.

\end{lem}

\begin{proof}

Suppose that $\bar{x}$ and $\bar{y}$ have length $m$.
Let $\gamma = 0$ and let $At$ be the computable set of all atomic formulas on the first $m$ variables in the language $L$.  Then \[\varphi^0(\bar{x},\bar{y}) = \bigwedge_{\phi\in At}(\phi(\bar{x}) \leftrightarrow \phi(\bar{y})),\]
which is finitary quantifier-free.  
Suppose we have determined the formulas $\varphi^\beta(\bar{x},\bar{y})$ for all $\beta < \gamma$ and all appropriate pairs of tuples of variables $\bar{x},\bar{y}$.  Then
\[\varphi^\gamma(\bar{x},\bar{y}) = \bigwedge_{\beta < \gamma}
[\bigwedge_{\bar{u},\bar{v}}(\forall\bar{u})(\exists \bar{v}) \varphi^\beta(\bar{x},\bar{u},\bar{y},\bar{v})\ \&\ 
\bigwedge_{\bar{v},\bar{u}}(\forall\bar{v})\big(\exists\bar{u})\varphi^\beta(\bar{x},\bar{u},\bar{y},\bar{v})],\]
which is a $\Pi^c_{2\gamma}$ formula.
\end{proof}

The next lemma is well-known, and the proof is straightforward (see \cite{AK}).      

\begin{lem}
\label{lem:equivalence:formulas}

Let $\mathcal{A}$ and $\mathcal{B}$ be structures for the same countable language, and let $\bar{a}$ and $\bar{b}$ be tuples of the same length, in $\mathcal{A}$ and $\mathcal{B}$, respectively. Then for any countable ordinal $\gamma$, if $(\mathcal{A},\bar{a}) \sim^\gamma (\mathcal{B},\bar{b})$, then the $\Sigma^c_\gamma$ formulas true of $\bar{a}$ in $\mathcal{A}$ are the same as those true of $\bar{b}$ in $\mathcal{B}$.

\end{lem}

\subsection{$\sim^\gamma$-equivalence in linear orderings}

In a linear ordering, the $\sim^\gamma$-class of a tuple $\bar{a}$ is determined by the $\sim^\gamma$-classes of the intervals with endpoints in $\bar{a}$.  Let $\mathcal{A}$ and $\mathcal{B}$ be linear orderings.  Let $\bar{a} = a_1 < \ldots < a_n$ be a tuple in $\mathcal{A}$, and let $\bar{b} = b_1 < \ldots < b_n$ be a tuple in $\mathcal{B}$.  Let $I_0,\ldots,I_n$ and $J_0,\ldots,J_n$ be the intervals in $\mathcal{A}$ and $\mathcal{B}$ determined by $\bar{a}$ and $\bar{b}$; i.e., $I_0$ is the interval $(-\infty,a_1)$ in $\mathcal{A}$, $J_0$ is the interval $(-\infty,b_1)$ in $\mathcal{B}$, for $i < n$, $I_i$ is the interval $(a_i,a_{i+1})$ in $\mathcal{A}$, $J_i$ is the interval $(b_i,b_{i+1})$ in $\mathcal{B}$, $I_n$ is the interval $(a_n,\infty)$ in $\mathcal{A}$, and $J_n$ is the interval $(b_n,\infty)$ in $\mathcal{B}$. The next lemma is also well-known, with a straightforward proof.  In \cite{AK}, there is a similar result, with the asymmetric back-and-forth relations $\leq_\gamma$ replacing the symmetric relations $\sim^\gamma$.    

\begin{lem}

$(\mathcal{A},\bar{a})\sim^\gamma(\mathcal{B},\bar{b})$ iff for $i\leq n$, $I_i\sim^\gamma J_i$.

\end{lem} 

\subsection{More on the orderings $L(G)$}     

We return to the orderings of form $L(G)$.  In the next subsection, we will prove that there do not exist $L_{\omega_1\omega}$ formulas that, for all $G$, interpret $G$ in $L(G)$.  Roughly speaking, the outline is as follows.  We assume that there are such formulas.  The formulas are $\Sigma_\alpha$, for some countable ordinal $\alpha$.  Moreover, they are $X$-computable $\Sigma_\alpha$ for some $X$ such that $\alpha < \omega_1^X$.  Taking $G$ to be the ordering $\omega_1^X$, we will produce tuples $\bar{b},\bar{c},\bar{b}'$ in $L(G)$ representing elements $a,e,a'$ of $G$ such that $\bar{b},\bar{c}\sim^\gamma \bar{c},\bar{b}'$, although in $G$, we have $a < e$ and $a' < e$.  This is a contradiction.  The current subsection gives several lemmas about the relations $\sim^\gamma$ on tuples in $L(G)$, and about automorphisms of $L(G)$.  These lemmas are what we need to produce the tuples $\bar{b},\bar{c},\bar{b}'$.             

To start off, we note that if $a_1,a_2 \sim^1 b_1,b_2$, then the sizes of the intervals $(a_1,a_2)$ and $(b_1,b_2)$ match.  Moreover, if $a \sim^2 b$, then $a$ and $b$ belong to maximal discrete intervals of the same size.  

\begin{lem}  
\label{lem:intervals:merge}
Let $I = (b,b')$, where $b < b'$, and let $J = (c,c')$, where $c < c'$.  Suppose $b\sim^\gamma c$ and $b'\sim^\gamma c'$, where some $b^*\in I$ and some $c^*\in J$ each have length $2$.  Then $I\sim^\gamma J$.

\end{lem}

\begin{proof} 

Suppose $\beta < \gamma$.  Take $\bar{d}$ in $I$.  We want $\bar{e}$ in $J$ such that \\
$(I,\bar{d})\sim^\beta(J,\bar{e})$.  We consider the cases $\beta = 0$, $\beta = 1$, and $\beta\geq 2$.  

\bigskip
\noindent
\textbf{Case 1}:  Suppose $\beta = 0$.  The fact that $J$ contains an element of length $2$ implies that it is an infinite interval.  We choose $\bar{e}$ in this interval ordered in the same way as $\bar{d}$.

\bigskip
\noindent
\textbf{Case 2}:  Suppose $\beta = 1$.  The tuple $\bar{d}$ partitions the interval $I = (b,b')$ into sub-intervals $I_0,\ldots,I_m$.  We need $\bar{e}$ partitioning $J$ into sub-intervals $J_0,\ldots,J_m$ of the same sizes.  The first few intervals $I_i$ may be finite.  Since $b\sim^2 c$, we can match these intervals.  Similarly, we can match the last few intervals, if these are finite.  For simplicity, we suppose that the intervals $I_0$ and $I_m$ are both infinite.  The tuple $\bar{d}$ is automorphic to a tuple $\bar{d}'$ lying entirely to the right of $c$, with first element infinitely far from $c$.  Let $d'$ be infinitely far to the right of the last term of $\bar{d}'$.  By Lemma \ref{lem10}, there is an automorphism of the interval $(c,\infty)$ taking $\bar{d}',d'$ to some $\bar{e},e'$ in the interval $(c,c')$.  We let the $J_i$'s be the sub-intervals of $J$ determined by $\bar{e}$.  These have the desired sizes.      

\bigskip
\noindent
\textbf{Case 3}:  Suppose $\beta\geq 2$.  We may suppose that $\bar{d} = \bar{d}_1,b^*,\bar{d}_2$.  The intervals $(b,\infty)$ and $(c,\infty)$ are $\sim^\gamma$-equivalent.  Therefore, we have $\bar{e}_1,c^{**}$ in $(c,\infty)$ $\sim^\beta$-equivalent to $\bar{d}_1,b^*$ in $(b,\infty)$.
Since $\beta \geq 2$, we have that $c^{**}$ has length $2$.
Let $p$ be the first term of $c$, let $r$ be the first term of $c^*$, and let $q$ be the first term of $c^{**}$.  Let $f$ be a permutation of $\mathbb{Q}$, preserving the order and the $A_i$'s, fixing $p$ and taking $q$ to $r$.  We have an automorphism $g$ of $(c,\infty)$ (or of $L(G)$) that takes $x\sigma$ to $f(x)\sigma$.  Let $\bar{e}_1'$ be $g(\bar{e}_1)$.  The sub-intervals of $I$ (or of $(b,\infty)$) determined by $\bar{d}_1,b^*$ are $\sim^\beta$ equivalent to the sub-intervals of $(c,\infty)$ determined by $\bar{e}_1,c^{**}$.  These are isomorphic to the sub-intervals determined by $\bar{e}_1',g(c^{**})$. Thus, the sub-intervals of $(b,\infty)$ determined by $\bar{d}_1,b^*$ are $\sim^\beta$-equivalent to the sub-intervals of $(c,\infty)$ determined by $\bar{e}_1',c^*$.   

In a similar way, we get $\bar{e}_2'$ such that the sub-intervals of $(c^*,\infty)$ determined by $c^*,\bar{e}_2'$ are $\sim^\beta$-equivalent to those determined by $b^*,\bar{d}_2$ in $(b^*,\infty)$.  We let $\bar{e}$ be $\bar{e}_1',\bar{e}_2'$.  All together, the sub-intervals of $(b,b')$ determined by $\bar{d}$ are $\sim^\beta$-equivalent to the corresponding sub-intervals of $(c,c')$ determined by $\bar{e}$.                                       
\end{proof}  

\begin{lem}
\label{lem18}

Let $\bar{b}_1,\bar{b}_2$, $\bar{c}_1,\bar{c}_2$ be increasing sequences in $L(G)$, where $\bar{b}_1\sim^\gamma\bar{c}_1$ and $\bar{b}_2\sim^\gamma\bar{c}_2$.  Suppose further that there is an element of length $2$ between the last element of $\bar{b}_1$ and the first element of $\bar{b}_2$, and there is an element of length $2$ between the last element of $\bar{c}_1$ and the first element of $\bar{c}_2$.  Then $\bar{b}_1,\bar{b}_2\sim^\gamma\bar{c}_1,\bar{c}_2$.  

\end{lem}

\begin{proof}  

Say that $\bar{b}_1 = (b_1,\ldots,b_k)$, $\bar{b}_2 = (b_{k+1},\ldots,b_n)$, $\bar{c}_1 = (c_1,\ldots,c_k)$, and
$\bar{c}_2 = (c_{k+1},\ldots,c_n)$.  Let $I_i$ be the intervals determined by $\bar{b}_1,\bar{b}_2$, and let $J_i$ be the intervals determined by $\bar{c}_1,\bar{c}_2$, for $i\leq n$.  The fact that $\bar{b}_1\sim^\gamma\bar{c}_1$ implies that $I_i\sim^\gamma J_i$ for $i < k$.  The fact that $\bar{b}_2\sim^\gamma\bar{c}_2$ implies that $I_i\sim^\gamma J_i$ for $k <i\leq n$.  It remains to show that $I_k\sim^\gamma J_k$.  We have $b_k\sim^\gamma c_k$ and $b_{k+1}\sim^\gamma c_{k+1}$.  We have elements of length $2$ in each of the intervals $I_k$ and $J_k$.  Applying the previous lemma, we get the fact that $I_k\sim^\gamma J_k$.  Therefore, $\bar{b}_1,\bar{b}_2\sim^\gamma\bar{c}_1,\bar{c}_2$.   
\end{proof}   

\begin{lem}
\label{lem:equivalence:shape}

Suppose $\bar{b}$, $\bar{b}'$ are tuples in $L(G)$ of the same shape.  Let $\bar{a}$, $\bar{a}'$ be the full tuples from $G$ mentioned by the $b_i$'s, or the $b'_i$'s.  If $\bar{a}\sim^\gamma\bar{a}'$, then $\bar{b}\sim^\gamma\bar{b}'$.  

\end{lem}

\begin{proof}

We proceed by induction on $\gamma$.  For $\gamma = 0$, the statement is trivially true.  Supposing that the statement holds for $\beta < \gamma$, we show it for $\gamma$.  Suppose $\bar{a}\sim^\gamma\bar{a}'$.  We will have $\bar{b}\sim^\gamma\bar{b}'$ provided that for all $\beta < \gamma$, 
  
\begin{enumerate}

\item  for any $\bar{d}$, there is some $\bar{d}'$ such that $\bar{b},\bar{d}\sim^\beta\bar{b}',\bar{d}'$, and 

\item  for any $\bar{d}'$, there is some $\bar{d}$ such that $\bar{b},\bar{d}\sim^\beta\bar{b}',\bar{d}'$.  

\end{enumerate}
By symmetry, it is enough to prove (1).  Say that $\bar{c}$ is the tuple of elements of $G$ mentioned in the 
$d_i$'s and not in $\bar{a}$.  Since $\bar{a}\sim^\gamma\bar{a}'$ in $G$, there is a tuple $\bar{c}'$ such that 
$\bar{a},\bar{c}\sim^\beta\bar{a}',\bar{c}'$.  In $L(G)$, we choose $\bar{d}'$, so that the ordering and shape of 
$\bar{b}',\bar{d}'$ matches that of $\bar{b},\bar{d}$, and for each $d'_i$, the tuple $\bar{a}',\bar{c}'$ mentioned in $d'_i$ corresponds to the one from $\bar{a},\bar{c}$ mentioned in $d_i$.  Using the fact that $\bar{b}'$ and 
$\bar{b}$ have the same shape, we can see that such $\bar{d}'$ exist.  By the induction hypothesis, we have 
$\bar{b},\bar{d}\sim^\beta\bar{b}'\bar{d}'$.     
\end{proof}  

\begin{defn}

We say that $\mathcal{A}$ is a computable infinitary substructure of $\mathcal{B}$ if $\mathcal{A}$ is a substructure of $\mathcal{B}$ and for all computable infinitary formulas $\varphi(\bar{x})$ and all $\bar{a}$ in 
$\mathcal{A}$, $\mathcal{B}\models\varphi(\bar{a})$ iff $\mathcal{A}\models\varphi(\bar{a})$.  (The definition is the same as elementary substructure except that the formulas are not elementary (finitary) first order.)  

\end{defn}

\begin{lem}
\label{lem:graphs:substructure}

Let $G_1$ and $G_2$ be directed graphs such that $G_1$ is a computable infinitary substructure of $G_2$.  Suppose also that $G_2$ is computable, so $L(G_2)$ is computable.  Then $L(G_1)$ is a computable infinitary substructure of $L(G_2)$.

\end{lem}

\begin{proof}

Note that $L(G_1)$ is a substructure of $L(G_2)$.  The Tarski-Vaught test was originally stated for elementary substructure, but it also works for computable infinitary substructure.  To show that $L(G_1)$ is a computable infinitary substructure of $L(G_2)$, it is enough to show that for any computable infinitary formula $\psi(\bar{x},u)$, if $L(G_2)\models\psi(\bar{b},d)$, where $\bar{b}$ is in $L(G_1)$, then $L(G_2)\models\psi(\bar{b},d')$ for some $d'\in L(G_1)$.
  
Say that $\psi$ is a $\Pi^c_\alpha$ formula.  Suppose $\bar{b}$ mentions $\bar{a}$ from $G_1$.  The tuple from $G_1$ mentioned by $d$ may include some elements from $\bar{a}$, plus some further elements $\bar{c}$.  By Lemma~\ref{lem:equivalence:characterization}, we have a computable infinitary formula $\varphi^\alpha_{\bar{a},\bar{c}}(\bar{x},\bar{y})$ defining in $G_2$ the $\sim^\alpha$-class of $\bar{a},\bar{c}$.  By the Tarski-Vaught test, there is some $\bar{c}'$ in $G_1$ such that $G_2 \models \varphi^\alpha_{\bar{a},\bar{c}}(\bar{a},\bar{c}')$.  Then in $G_2$, $\bar{a},\bar{c} \sim^\alpha \bar{a},\bar{c}'$.  Say that $\bar{u}$ is the tuple in $G_2$ mentioned by $d$.  Each $u_i$ is in $\bar{a},\bar{c}$.  Let $\bar{v}$ be the tuple in $G_1$ such that if $u_i\in \bar{a}$, then $v_i = u_i$ and if $u_i\in\bar{c}$, then $v_i$ is the element of $\bar{c}'$ corresponding to $\bar{c}$.  Thus, $\bar{u}\sim^\alpha\bar{v}$. We choose $d'$, mentioning the tuple $\bar{v}$, such that $\bar{b},d$ and $\bar{b},d'$ have the same ordering and the same shape.  Then by Lemma~\ref{lem:equivalence:shape}, $\bar{b},d\sim^\alpha\bar{b},d'$.  By Lemma~\ref{lem:equivalence:formulas}, we conclude that $L(G_2)\models\psi(\bar{b},d')$, as required.             
\end{proof}

\subsection{Proof of Theorem \ref{main}}

Theorem \ref{main} says that there are no $L_{\omega_1\omega}$-formulas that, for all directed graphs $G$, define an interpretation of $G$ in $L(G)$.  We introduce the ideas of the proof in a warm-up result.  Among the directed graphs are the linear orderings.  The Harrison ordering $H$ has order type $\omega_1^{CK}(1+\eta)$.  While $\omega_1^{CK}$ has no computable copy, $H$ does have a computable copy.  It is well known that $H$ and $\omega_1^{CK}$ satisfy the same computable infinitary sentences.  In fact, they satisfy the same $\Pi_\alpha$ sentences of $L_{\omega_1\omega}$ for all computable ordinals $\alpha$. 

Let $I$ be the initial segment of $H$ of order type $\omega_1^{CK}$.  Thinking of $H$ as a directed graph, we can form the linear orderings $L(H)$ and $L(I)$.
By Proposition \ref{pr:comp-effective-int}, just because $H$ has a computable copy, it is effectively interpreted in every structure $\mathcal{B}$.
Our warm-up result will say that there are no computable infinitary formulas that define an interpretation of $H$ in $L(H)$ and also define an interpretation of $I$ in $L(I)$. 
\begin{prop}
  
$L(I)$ is a computable infinitary substructure of $L(H)$.

\end{prop}

\begin{proof}

Since $I$ and $H$ satisfy the same computable infinitary sentences and every element of $I$ is defined by a computable infinitary formula, it follows that $I$ is a computable infinitary substructure of $H$.  We apply Lemma~\ref{lem:graphs:substructure} to conclude that $L(I)$ is a computable infinitary substructure of $H$.
\end{proof}

\begin{prop} [Warm-up]  
\label{prop12} 

There do not exist computable infinitary formulas that define an interpretation of $H$ in $L(H)$ and also define an interpretation of $I$ in $L(I)$.  

\end{prop}      

\begin{proof}

Suppose there are computable infinitary formulas that define an interpretation of $H$ in $L(H)$, and also define an interpretation of $I$ in $L(I)$.  Say $D$, $\sim$, and \textcircled{$<$} are the sets of tuples defined by these formulas in $L(H)$.  We note that all elements of $I$ are represented by tuples from $D$ that are in $L(I)$, and all tuples from $D$ that are in $L(I)$ represent elements of $I$.  We can translate computable infinitary formulas describing $H$ and its elements into computable infinitary formulas about tuples in $L(H)$, referring to the formulas that define $D$, $\sim$, and \textcircled{$<$}.   

For each computable ordinal $\alpha$, we have a formula $\varphi_\alpha(x)$ saying of an element $x$ in $H$ that $pred(x) = \{y:y < x\}$ has order type $\alpha$.  Let $\psi_\alpha(\bar{x})$ be the translation formula saying of a tuple $\bar{x}$ that it is in $D$ and the set of predecessors of the equivalence class of $\bar{x}$ has order type $\alpha$.  For each computable ordinal $\alpha$, there is a tuple in $D$ satisfying $\psi_\alpha(\bar{x})$ (for an appropriate $\bar{x}$).  Since $L(I)$ is a computable infinitary substructure of $L(H)$, some tuple from $D$ in $L(I)$ also satisfies $\psi_\alpha(\bar{x})$.  Moreover, each tuple from $D$ in $L(I)$ satisfies one of the formulas  $\psi_\alpha$.  Recall that the ordering $H$ is computable, and so is $L(H)$.  
We define equivalence relations $\equiv^\gamma$ on $D$.

\begin{defn}

For tuples $\bar{a}$ and $\bar{b}$ in $D$, let $\bar{a}\equiv^\gamma\bar{b}$ iff 

\begin{enumerate}

\item  $\bar{a}$ and $\bar{b}$ have the same shape and 

\item  $\bar{a}\sim^\gamma\bar{b}$.  

\end{enumerate}

\end{defn}

\noindent
\textbf{Fact}:  For each computable ordinal $\gamma$ and each $\bar{a}$ in $D$, the $\equiv^\gamma$-class of $\bar{a}$ is defined by a computable infinitary formula.  

\bigskip

We need one more lemma.

\begin{lem}
\label{stronger}

For each computable ordinal $\gamma$, there is a $\equiv^\gamma$-class $C$ such that there are arbitrarily large computable ordinals $\alpha$ for which some $\bar{b}$ in $C$ satisfies $\psi_\alpha$.    

\end{lem}

\begin{proof}

In $L(H)$, we have a tuple $\bar{b}$ in $D$ not satisfying any of the formulas $\psi_\alpha$ for computable ordinals $\alpha$.  Let $C$ be the $\equiv^\gamma$-class of $\bar{b}$.  Since $L(I)$ is a computable infinitary substructure of $L(H)$, and $C$ is defined by a computable infinitary formula, we must have tuples of $L(I)$ belonging to $C$ and satisfying $\psi_\alpha$ for arbitrarily large computable ordinals~$\alpha$.  
\end{proof}     

Suppose that the formulas defining $D$, \textcircled{$<$}, and $\sim$ are all $\Sigma^c_\gamma$.  Since $D$ may have no fixed arity, we mean that there is a computable sequence of $\Sigma^c_\gamma$ formulas defining the sets of $n$-tuples in $D$, and similarly for \textcircled{$<$} and $\sim$.  By Lemma~\ref{stronger}, there is a set $C\subseteq D$ in which all tuples have the same shape and are in the same $\sim^\gamma$-class---in particular, the tuples in $C$ all have the same arity.  We choose tuples $\bar{b}$ and $\bar{c}$ in $L(I)$, both belonging to $C$, such that $\bar{b}$ satisfies $\psi_\alpha$ and $\bar{c}$ satisfies $\psi_\beta$, where $\alpha < \beta$.  

By Lemma~\ref{lem11}, we may suppose that all elements of the tuple $\bar{b}$ lie to the left of the $<$-first element of $\bar{c}$, and the interval between the $<$-greatest element of $\bar{b}$ and the $<$-first element of $\bar{c}$ contains an element of length $2$.  Also, by the same lemma, we have a tuple $\bar{b}'$, automorphic to $\bar{b}$, such that all elements of $\bar{b}'$ lie to the right of the $<$-greatest element of $\bar{c}$, and the interval between the $<$-greatest element of $\bar{c}$ and the $<$-first element of $\bar{b}'$ contains an element of length $2$.  Since $\bar{b}$ satisfies $\psi_\alpha$ and $\bar{c}$ satisfies $\psi_\beta$, we should have $L(I) \models \bar{b}\textcircled{$<$}\bar{c}$.  Since $\bar{b}'$ is automorphic to $\bar{b}$, it should also satisfy $\psi_\alpha$, so we should have $L(I) \models \bar{b}'\textcircled{$<$}\bar{c}$.  Applying Lemma \ref{lem18}, we get the fact that $\bar{b},\bar{c}\sim^\gamma \bar{c},\bar{b}'$.  Therefore, since $L(I) \models \bar{b}\textcircled{$<$}\bar{c}$, and $\textcircled{$<$}$ is defined by a $\Sigma^c_\gamma$-formula, we have $L(I) \models \bar{c}\textcircled{$<$}\bar{b}'$.  This is the contradiction that we were expecting when we set out to prove Proposition \ref{prop12}.    
\end{proof} 

We have proved Proposition \ref{prop12}, saying that there do not exist computable infinitary formulas that define an interpretation both for the Harrison ordering $H$ in $L(H)$ and for the well-ordered initial segment $I$ in $L(I)$.  We assumed that there were computable infinitary formulas, say $\Sigma^c_\gamma$, defining both interpretations, and we arrived at a contradiction.  We used $H$ and $L(H)$ to arrive at a sequence of tuples $\bar{b}_\alpha$ in $L(I)$, representing arbitrarily large elements of $I$, and all having the same shape and satisfying the same computable $\Sigma^c_\gamma$ formulas.  We then used automorphisms of $L(I)$ to show that our proposed interpretation failed.  The next result says that, in fact, there do not exist computable infinitary formulas that define an interpretation for $I$ in $L(I)$.  Of course, $I$ is isomorphic to $\omega_1^{CK}$.                 

\begin{prop}
\label{prop6}

There is no interpretation of $\omega_1^{CK}$ in $L(\omega_1^{CK})$ defined by computable infinitary formulas.

\end{prop}  

\begin{proof} Suppose we have an interpretation of $\omega_1^{CK}$ in $L(\omega_1^{CK})$, defined by computable infinitary formulas.  Say that the formulas that define the appropriate $D$, \textcircled{$<$}, and $\sim$ are $\Sigma^c_\gamma$.  Our assumption gives the fact that for a Harrison ordering with well-ordered initial segment $I$, these formulas interpret $I$ in $L(I)$.  However, the assumption does not say that they also interpret $H$ in $L(H)$.  Thus, we are not in a position to use the important Lemma \ref{stronger}.       

\bigskip

The following lemma is simple enough that we omit the proof.            

\begin{lem}

Let $\mathcal{A}$ be a computable structure.  If $\mathcal{B}$ satisfies the computable infinitary sentences true in $\mathcal{A}$, then the formulas $\varphi^\gamma_{\bar{d}}$ that define the $\sim^\gamma$-equivalence classes of all tuples in $\mathcal{A}$ also define the $\sim^\gamma$-equivalence classes of all tuples in $\mathcal{B}$.  Moreover, if $\mathcal{B}\models\varphi^\gamma_{\bar{d}}(\bar{b})$, then the $\Sigma^c_\gamma$-formulas true of $\bar{b}$ in $\mathcal{B}$ are the same as those true of $\bar{d}$ in $\mathcal{A}$.       

\end{lem} 

The next lemma gives the conclusion of Lemma \ref{stronger}.  The proof involves locating 
$\omega_1^{CK}$ inside a larger ordering similar to the Harrison ordering.     

\begin{lem}

In $L(\omega_1^{CK})$, there are tuples $\bar{d}_\alpha$, corresponding to arbitrarily large computable ordinals $\alpha$, such that all $\bar{d}_\alpha$ are in $D$, all have the same length and shape, all are $\sim^\gamma$-equivalent, and $\bar{d}_\alpha$ satisfies $\psi_\alpha$.

\end{lem}

\begin{proof} [Proof of lemma]

We use Barwise-Kreisel Compactness.  Let $\Gamma$ be a $\Pi^1_1$ set of computable infinitary sentences describing a structure
\[\mathcal{U} = (U_1\cup U_2,U_1,<_1,U_2,<_2,F,c)\] 
such that

\begin{enumerate}

\item $U_1$ and $U_2$ are disjoint sets, 

\item  $(U_1,<_1)$ is a linear ordering that satisfies the computable infinitary sentences true in 
$\omega_1^{CK}$ and $H$---since $H$ is computable, this is $\Pi^1_1$,

\item  $(U_2,<_2)$ satisfies the computable infinitary sentences true in $L(\omega_1^{CK})$---this is $\Pi^1_1$ since $L(H)$ is computable and $L(I)$ is a computable infinitary substructure of $L(H)$,

\item  $F$ is a function from $D^{U_2}$ to $U_1$ that induces an isomorphism between $(D^{U_2},\textcircled{$<$})/_{\sim^{U^2}}$ and $(U_1,<_1)$,

\item  $c$ is a constant in $U_1$ such that $c >_1 \alpha$ for all computable ordinals $\alpha$; i.e., there is a proper initial segment of $<_1$-$pred(c)$ of type $\alpha$.  

\end{enumerate}

Every $\Delta^1_1$ subset of $\Gamma$ is satisfied by taking copies of $\omega_1^{CK}$, $L(\omega_1^{CK})$, with an appropriate function $F$, and letting $c$ be a sufficiently large computable ordinal.  Therefore, the whole set $\Gamma$ has a model.  Let $\bar{b}$ be an element of $D^{U_2}$ such that $F(\bar{b}) = c$.
Let $C$ be the set of tuples of $U_2$ having the shape of $\bar{b}$ and $\sim^\gamma$-equivalent to $\bar{b}$.
Since $(U_2,<_2)$ satisfies the same computable infinitary sentences true in the computable structure $L(H)$,
by the lemma above, the $\sim^\gamma$-equivalence class of $\bar{b}$ is defined in $(U_2,<_2)$ by a computable infinitary formula.  For each computable ordinal $\alpha$, we have a computable infinitary sentence $\chi_\alpha$ saying that some tuple in $C$ does not satisfy $\psi_\beta$ for any $\beta < \alpha$.  The sentence $\chi_\alpha$ is true in our model of $\Gamma$, witnessed by $\bar{b}$ such that $F(\bar{b}) = c$.  Therefore, the sentence $\chi_\alpha$ is true also in $L(\omega_1^{CK})$, witnessed by some $\bar{b}'$.  Since our formulas define an interpretation of $\omega_1^{CK}$ in $L(\omega_1^{CK})$, the witness $\bar{b}'$ for $\chi_{\alpha}$ in $L(\omega^{CK}_1)$ must satisfy $\psi_\gamma$ for some $\gamma\geq\alpha$.
\end{proof}

Now, we can proceed as in the proof of Proposition \ref{prop12}.  We are working in $L(\omega_1^{CK})$.  We choose $\bar{b},\bar{c}$, from the sequence of $\bar{d}_\alpha$'s in the lemma, such that $\bar{b}\sim^\gamma\bar{c}$, where $\bar{b}$ satisfies $\psi_\alpha$ and $\bar{c}$ satisfies $\psi_\beta$, for $\alpha < \beta$.  By Lemma \ref{lem11}, we may suppose that the elements of $\bar{b}$ all lie to the left of the $<$-first element of $\bar{c}$, and the interval between the $<$-greatest element of $\bar{b}$ and the $<$-first element of $\bar{c}$ contains an element of length~$2$.  Since $\alpha < \beta$, we should have $L(\omega_1^{CK}) \models \bar{b} \textcircled{$<$} \bar{c}$.  We can take $\bar{b}'$ automorphic to $\bar{b}$ such that all elements of $\bar{b}'$ lie to the right of the $<$-greatest element of $\bar{c}$, and the interval between the $<$-greatest element of $\bar{c}$ and the $<$-first element of $\bar{b}'$ contains an element of length $2$.  Clearly, $L(\omega_1^{CK}) \models \bar{b}' \textcircled{$<$} \bar{c} $ since $\bar{b}'$ satisfies $\psi_\alpha(\bar{x})$.  Applying Lemma~\ref{lem18} we get the fact that $\bar{b},\bar{c} \sim^\gamma \bar{c},\bar{b}'$.  It follows that $L(\omega_1^{CK}) \models \bar{c} \mbox{\textcircled{$<$}} \bar{b}'$, which is a contradiction.
\end{proof}

We are ready to complete the proof of Theorem \ref{main}, saying that there is no tuple of $L_{\omega_1\omega}$-formulas that, for all directed graphs $G$, interprets $G$ in $L(G)$.  

\begin{proof} [Proof of Theorem \ref{main}]

Suppose that we have such formulas.  For some $X$, the formulas are $X$-computable infinitary.  Let $G$ be a linear ordering of type $\omega_1^X$. 
 Relativizing Proposition \ref{prop6}, we have the fact that $G$ is not interpreted in $L(G)$ by any $X$-computable formulas.  
\end{proof} 

The Friedman-Stanley embedding represents a uniform effective encoding of directed graphs in linear orderings.  We have seen that there is no uniform interpretation of the input graph in the output linear ordering.       

\begin{conj}

Let $\Phi$ be a Turing computable embedding of directed graphs in linear orderings.  There do not exist $L_{\omega_1\omega}$ formulas that, for all directed graphs $G$, define an interpretation of $G$ in $\Phi(G)$.     

\end{conj}

\end{document}